\documentclass[12pt,a4paper,reqno,oneside]{amsart}
\usepackage{t1enc}
\usepackage{times}
\usepackage[mathscr]{euscript}
\usepackage{graphicx}

\usepackage{hyperref}

\usepackage{amsmath, amsthm, amsfonts, amssymb}

\usepackage{xcolor}

\theoremstyle{plain}
\newtheorem{thm}{Theorem}[section]
\newtheorem*{thm*}{Theorem}
\newtheorem{lem}{Lemma}[section]

\newtheorem{corl}{Corollary}[section]
\theoremstyle{definition}

\newtheorem{example}{Example}
\theoremstyle{remark}
\newtheorem{remk}{Remark}[section]

\newcommand{\1}{1\!\!\,{\rm I}}

\newcommand{\wt}{\widetilde}

\newcommand{\be}{\begin{equation}}
\newcommand{\ee}{\end{equation}}

\newcommand{\ve}{\varepsilon}
\newcommand{\vf}{\varphi}

\newcommand{\mbR}{{\mathbb R}}
\newcommand{\mbN}{{\mathbb N}}

\newcommand{\sign}{\mathop{\mathrm{sign}}}

\newcommand{\cF}{{\mathcal F}}

\newcommand{\supp}{\supp{\rm supp}}

\newcommand{\Pb}{\mathrm{P}}

\begin{document}
\title{On perturbations of an ODE with non-Lipschitz  coefficients by a  small self-similar noise
}

\author{Andrey  Pilipenko}
\address{
 Institute of Mathematics of Ukrainian National Academy of Sciences; National Technical University of Ukraine ``KPI''}
 \thanks{Research is partially supported by FP7-People-2011-IRSES Project number 295164}
\author{ Frank Norbert Proske}
\address{
 Department of Mathematics, University of Oslo
 }

\begin{abstract}
We study the  limit behavior
 of differential equations with non-Lipschitz coefficients that are perturbed by a small self-similar noise.
It is proved that the limiting process is equal to the maximal  solution or minimal solution with certain probabilities $p_+$ and $p_-=1-p_+$, respectively.
 We propose a space-time
transformation that reduces  the investigation of the original problem to the study of the exact growth rate of a solution to a certain SDE with self-similar noise. This problem 
  is interesting in itself. Moreover, the probabilities $p_+$ and $p_-$ coincide 
with probabilities that the solution of the transformed equation converges to $+\infty$  or $-\infty$  as $t\to\infty,$ respectively.
\end{abstract}
\keywords{zero-noise limit; Peano phenomenon;    growth rate of solutions to stochastic equations; selection of a solution; self-similar process}
\subjclass[2010]{60H10, 	60F99}

 \maketitle

\section*{Introduction}
We study the limit behavior of the sequence  of small-noise stochastic equations
\be\label{eq_001}
X_\ve(t)=\int_0^t (c_+\1_{X_\ve(s)\geq 0} -c_-\1_{X_\ve(s)< 0})
 |X_\ve(s)|^\alpha ds + \ve B_\beta(t),
\ee
where $c_\pm>0, \alpha\in(-1,1),$ $B_\beta(t), t\geq 0$ is a self-similar process with index $\beta>0.$

Note that the drift $a(x)=(c_+\1_{x\geq 0} -c_-\1_{x<0})|x|^\alpha \sign x$ does not satisfy the Lipschitz property at $0$, and the limit equation 
\be\label{eq002}
X_0(t)=\int_0^t a(X_0(s)) ds, t\geq 0,
\ee
has two families of solutions
$$
X_0^{+,\tau}(t) = 
\begin{cases} 0,\ t\in[0,\tau],\\
(c_+(1-\alpha)t)^{\frac{1}{1-\alpha}}, \ t\geq \tau;
\end{cases}
X_0^{-,\tau}(t) = 
\begin{cases} 0,\ t\in[0,\tau],\\
-(c_-(1-\alpha)t)^{\frac{1}{1-\alpha}}, \ t\geq \tau,
\end{cases}
$$
where $\tau\in[0,\infty]$ is a parameter.

{Note that if the  self-similar process is a L\'{e}vy stable process or a fractional Brownian motion (fBm), then there is a unique solution to
the  equation \eqref{eq_001} under a suitable  relation between $\alpha $ and $\beta$, see \cite{EngelbertSchmidt, TanakaTsuchiya74, NualartOuknine, Portenko1994} and Remark \ref{remUniq} below. 
That's why a limit of $\{X_\ve\}$ as $\ve \to 0$ can be considered as a natural selection for a solution to \eqref{eq002}. }

The seminal  result on a selection problem was a paper of Bafico and Baldi \cite{BaficoBaldi} who considered the case where $B_\beta$ is a Wiener process. They proved (even for  more general form of the drift) that the sequence of distributions of $X_\ve$ weakly converges as $\ve \to0$ to 
$p_-\delta_{X^-}+p_+\delta_{X^+}$, where $p_-+p_+=1$ and $X^\pm$ are the maximal  and the minimal 
solutions to the limit equation. In our case $X^\pm(t)=\pm(c_\pm(1-\alpha)t)^{\frac{1}{1-\alpha}}$.

The small-noise problem for non-Lipschitz equations was studied from the different points of view,
see \cite{Flandoli}, \cite{FlandoliMetrika} and references therein. As for the study of the small-noise problem for SDE's with discontinuous multidimensional drift vector fields we also refer to \cite{DelarueFlandoliVincenzi}, \cite{BuckdahnOuknineQuincampoix} and \cite{PilipenkoProske}.   

We propose a new approach. We make the following transformation of space and time 
$$\tilde X (t):=\tilde X_\ve(t):=\ve^{\frac{1}{ (1-\alpha)\beta-1}}   X_\ve (\ve^{\frac{\alpha-1}{ (1-\alpha)\beta-1}}t),\ t\geq 0,$$
then  we use the self similarity property of the process $B_\beta,$ and show that $\tilde X$ satisfies SDE
\be\label{eq_dod}
\wt X(t)=\int_0^t (c_+\1_{\wt X (s)\geq 0} -c_-\1_{\wt X (s)< 0})
 |\wt X (s)|^\alpha ds + \wt B_\beta(t), t\geq 0,
\ee
where $\wt B_\beta\overset{d}=  B_\beta.$ Observe that if \eqref{eq_dod} has a unique solution, then its distribution is independent of $\ve$.
 Hence, the limit behavior of $X_\ve$ as $\ve \to0$ is closely related with the behavior of $\tilde X(t)$ as $t\to\infty.$

Under general conditions a self-similar process with parameter $\beta$ a.s. has a growth that does not exceed 
$t^{\beta+\delta}$  as $t\to\infty$ for any $\delta>0$ (see \cite{Khintchine,  Kono, Takashima89}). This holds, for example, if $\tilde B$ is a L\'{e}vy $\beta^{-1}$-stable process or an fBm.
So, if $\alpha +\beta^{-1}>1$, then $\tilde B_\beta(t)=o(|X^\pm(t)|), t\to\infty, $ where
 $X^\pm=\pm(c_\pm(1-\alpha)t)^{\frac{1}{1-\alpha}}$. Therefore it is natural to 
expect that if $\wt X(t)$ converges to $+\infty$ (or $-\infty$) as $t\to\infty $ and  $\alpha +\beta^{-1}>1,$ then 
$\wt X(t)\sim X^+(t), t\to\infty$ (or $\sim X^-(t) $ respectively). We prove such result 
on asymptotic behavior of (small deterministic) perturbation of an integral equation in \S\ref{section3}. 
It should be noted that assumption $\alpha+\beta^{-1}>1$ ensures existence of a unique solution 
if $B_\beta$ is a symmetric $\beta^{-1}$-stable  L\'{e}vy process or fBm, see Remark \ref{remUniq}. Moreover uniqueness fails  if $\alpha+\beta^{-1}<1$, see \cite[Theorem 3.2]{TanakaTsuchiya74} for a L\'{e}vy stable process (the proof of non-uniqueness for fBm is the same).

The application of deterministic results of \S \ref{section3} to the study of exact growth rate of solutions to the  SDE \eqref{eq_dod} is given in \S \ref{sec4}. 
This question is interesting by itself. 
 {It also  may have an application in the construction of consistent estimators for parameters $\alpha$ or $c_\pm$
of a solution of equation \eqref{eq_dod}.}  
The first results on    exact growth rates of solutions to stochastic differential equations were obtained  by Gikhman and Skorokhod \cite{GS} who considered SDEs with Wiener noise,
 see also \cite{KKR, BKT}. We cover some of their results, and our method does not use the It\^{o}  formula. 
Note that cases $\alpha>0$ and $\alpha<0$ are considered separately (here and also in \cite{GS}). This is related with monotonicity properties of the function $x^\alpha$  and hence with the properties of the limit ODE.

The problem whether $|\tilde X(t)|\to\infty$ as $t\to\infty$ a.s. is significant, it also appeared in \cite{GS}. We prove  {almost sure convergence to infinity} 
in the case when $B_\beta$ is a L\'{e}vy  stable process and an fBm in \S \ref{Examples}, see Examples \ref{ex1} and \ref{ex2}. 

The small-noise problem for L\'{e}vy $\beta$-stable process was solved by Flandoli and Hoegele \cite{FlandoliHoegele} who even find limit probabilities $p_\pm.$ Their method was based on very careful study of jumps of the L\'{e}vy process and exits from the neighborhood of 0.

Our approach  allows us to study the small-noise problem even in the case when 
$X_\ve(0)=x_\ve\to0, \ve\to 0.$ Under suitable normalization of $x_\ve$, the corresponding  limit
 probabilities $p_\pm$ may be different from \cite{FlandoliHoegele}. Our approach may be useful
 in the study of the small-noise problem where a drift behaves at 0 ``similarly'' to $(c_+\1_{x\geq 0} -c_-\1_{x<0})|x|^\alpha$, or in the case when   the noise a L\'{e}vy process whose   L\'{e}vy measure in a neighborhood of zero is ``close'' to L\'{e}vy measure of   L\'{e}vy $\beta$-stable process. We also plan to study the multiplicative  noise.

\section{Main Result}\label{sec1}

Let $B_\beta(t), t\geq 0,$
be a self-similar process with index $\beta>0,$ i.e., for any $a>0$ the distribution of processes
$\{B_\beta(at), t\geq 0\}$ and $\{ a^\beta B_\beta(t), t\geq 0\}$  are equal. We will always assume that all considered processes have c\'{a}dl\'{a}g trajectories.

Consider stochastic equation
\be\label{eq_main_0}
X_\ve(t)=\int_0^t (c_+\1_{X_\ve(s)\geq 0} -c_-\1_{X_\ve(s)< 0})
 |X_\ve(s)|^\alpha ds + \ve B_\beta(t),\ t\geq 0,
\ee
where $c_\pm>0.$

By a weak solution to \eqref{eq_main_0} we call a pair of adapted processes
$(\tilde X, \tilde B_\beta)$ 
defined on a filtered probability space $(\Omega, \cF, P, \{\cF_t, t\geq 0\})$ such that

1) $(\tilde X, \tilde B_\ve)$ satisfy \eqref{eq_main_0} a.s.,

2) $\tilde B_\beta \overset{d}{=} B_\beta$.

Further  we will omit tilde in the notation and assume that all processes are defined on the corresponding probability space.

If $B_\beta$ is a L\'{e}vy process, then we naturally assume that $\tilde B_\beta$ is $\cF_t$-L\'{e}vy process.
If  $B_\beta$ is an fBm, we make the same assumption on $\cF_t$ as in \cite{NualartOuknine}.

The weak  uniqueness means that all weak solutions have the same distribution.

\begin{thm}\label{thm_main}
Let  $\alpha+\beta^{-1}>1.$ Assume that 
\begin{enumerate} 
\item there exists a unique weak solution to the equation
\be\label{eq_main_00}
X(t)=\int_0^t (c_+\1_{X(s)\geq 0} -c_-\1_{X(s)< 0})
 |X(s)|^\alpha ds +  B_\beta(t),\ t\geq 0;
\ee
\item $\Pb\left(\lim_{t\to\infty} |X(t)|=\infty \right)=1$;
\item $\alpha\in(0,1)$ and there exists a function $h(t)=o(t^{\frac{1}{1-\alpha}}), t\to\infty,$ and a random sequence $\{t_n\}, \lim_{n\to\infty}t_n=\infty,$
such that $|B_\beta(t)- B_\beta(t_n)|\leq h(t-t_n), t\geq t_n,$ a.s.

or

$\alpha\in(-1,0]$ and $B_\beta(t)=o(t^{\frac1{1-\alpha}}), t\to\infty$ a.s.
\end{enumerate}
Then the sequence of distributions of processes $\{X_\ve\}$ converges in $D([0,\infty))$ to the distribution
of the process 
$$
((1-\alpha)t)^{\frac{1}{1-\alpha}}\left({c_+^{\frac{1}{1-\alpha}}}\1_{\{\lim_{t\to\infty}X(t)=+\infty\}}- {c_-^{\frac{1}{1-\alpha}}}\1_{\{\lim_{t\to\infty}X(t)=-\infty\}}\right).  
$$
 \end{thm}
 \begin{remk}
The limit process  equals  $X^+(t)=(c_+(1-\alpha)t)^{\frac{1}{1-\alpha}}$ or $X^-(t)=-(c_-(1-\alpha)t)^{\frac{1}{1-\alpha}}$  with probabilities $$p_\pm=\Pb(\lim_{t\to\infty}X(t)=\pm\infty),$$
respectively.  Functions $X^\pm(t), t\geq0$, are the
 maximal positive solution to the ODE 
$$
y'(t)=c_+ y^\alpha(t), t\geq 0, \ y(0)=0,
$$
and  the minimal negative solution to the ODE
  $$y'(t)=-c_- |y|^\alpha(t), t\geq 0,\  y(0)=0,$$ 
  respectively.
\end{remk}
\begin{remk}\label{remUniq} Assume that $B_\beta$ is a symmetric $\beta^{-1}$ L\'{e}vy stable process and 
 $\alpha+\beta^{-1}>1.$ 
Then equation \eqref{eq_main_00} has a unique (strong) solution.
The case $\alpha>0$ was proved in 
 \cite{TanakaTsuchiya74}. The existence of a  weak solution for $\alpha\leq 0$ follows from \cite{Portenko1994},
 uniqueness was  proved in \cite{ChenWang}.

If $B_\beta$ is a Brownian motion, then $\beta=1/2$, and, see for example \cite{EngelbertSchmidt}, equation
\eqref{eq_main_00} has a unique solution if $\alpha>-1,$ i.e., $\alpha+(1/2)^{-1}>1$.

If $B_\beta$ is a fractional Brownian motion with Hurst parameter $H=\beta$, then  (strong) existence and  uniqueness 
holds for $H\in (0,1)$ and $\alpha\geq 0$, see \cite{NualartOuknine}.
 If $H\in (1,2)$,  it was 
proved  in \cite{NualartOuknine} that (strong) existence uniqueness holds for $\alpha+0.5H^{-1}>1.$ However
their method works for $\alpha+H^{-1}>1.$ Indeed, 
it follows from their estimates that (see \cite[page 110]{NualartOuknine} with their notations)
 $$
\forall \ve>0\ \exists K_1, K_2>0\ \ \int_0^T\beta^2(s)ds\leq\dots\leq K_1 \int_0^T  s^{1-2H+2\alpha(H-\ve)} ds G^{2\alpha}=
$$
$$
K_2 G^{2\alpha}<\infty
$$
  if $1-2H+2\alpha(H-\ve)>-1,$  that is, $\alpha+ (H-\ve)^{-1}>\frac{H}{H-\ve}$.
  The bound $\int_0^T\beta^2(s)ds\leq
K_2 G^{2\alpha}$ was sufficient for existence and uniqueness. Since $\ve>0$ was arbitrary, we have existence and uniqueness  if $\alpha+H^{-1}>1.$
\end{remk}
\begin{remk}\label{remkCond3} Condition 3 of the Theorem is satisfied for a L\'{e}vy stable process or a fractional Brownian motion. Indeed, it follows from \cite{Khintchine,  Kono} that for any $\ve>0$
$$
\lim_{t\to\infty}\frac{B_\beta(t)}{t^{\beta+\ve}}=0\ \ \ \mbox{a.s.}
$$
Let  $M>0$ and $\ve\in (0, \frac{1}{1-\alpha}-\beta)$. Set 
$$
h_{M,\ve}(t):= M+t^{\beta+\ve} =o(t^{\frac{1}{1-\alpha}}), \ t\to\infty,
$$
$$
A_{M,\ve}=\{\forall t\geq 0\ :\ |B_\beta(t)|\leq h_{M,\ve}(t)\}.
$$
 It follows from the Poincar$\acute{e}$  recurrence theorem that 
\be\label{eq230}
\Pb\left(A_{M,\ve}\cap\{\exists N\ \forall n\geq N\ \exists t\geq 0\ :\ |B_\beta(t+n)-B_\beta(n)|\geq h_{M,\ve}(t)\}\right)=0.
\ee
Since $\lim_{M\to+\infty}\Pb(A_{M,\ve})=1,$ \eqref{eq230} implies condition 3 of the Theorem.

Verification of condition 2 for these noises is given in \S\ref{sec4}.
\end{remk}

\section{Transformations of the main equation}\label{sec2}

Let $B_\beta(t), t\geq 0$
be a self-similar process with index $\beta>0.$

By   $X_{x,c,\alpha,\ve}(t)$ denote a (weak) solution of 
\be\label{eq_main}
X(t)=x+\int_0^t (c_+\1_{X(s)\geq 0} -c_-\1_{X(s)< 0})
 |X(s)|^\alpha ds + \ve B_\beta(t).
\ee
 
\begin{lem}
Let $X$
 be a solution of \eqref{eq_main}. Then the process $\wt X(t):=\ve^\delta X(\ve^{-\gamma}t)$ is a solution of
$$
\wt X(t)=x\ve^\delta+ \ve^{\delta(1-\alpha)-\gamma}\int_0^t (c_+\1_{\wt X(s)\geq 0} -c_-\1_{\wt X(s)< 0})
 |\wt X(s)|^\alpha  ds +
$$
\be\label{eq_main1}
\ve^{1+\delta-\gamma\beta} \wt B_\beta(t), t\geq 0,
\ee
where $\wt B_\beta\overset{d}{=}  B_\beta.$

That is, the distributions of 
\newline
\centerline{$\ve^\delta X_{x,c,\alpha,\ve}(\ve^{-\gamma}t), t\geq 0, $ and $X_{x \ve^\delta ,c \ve^{\delta(1-\alpha)-\gamma},\alpha,\ve^{1+\delta-\gamma\beta}}(t), t\geq 0,$}
 \newline
are equal 
if at least one of equations \eqref{eq_main} or  \eqref{eq_main1} has a unique weak solution.

In particular, if $x=0,\  \delta=\frac{1}{(1-\alpha)\beta-1},\ \gamma= \frac{1-\alpha}{(1-\alpha)\beta-1}    $, then $\wt X(t)$ satisfies the equation  {\eqref{eq_dod}.}
\end{lem}
\begin{proof}
$$
\wt X(t)=\ve^\delta X(\ve^{-\gamma}t)=
x\ve^\delta + \ve^\delta \int_0^{\ve^{-\gamma}t } (c_+\1_{X(s)\geq 0} -c_-\1_{X(s)< 0})
 |X(s)|^\alpha  ds +
$$
$$
\ve^\delta \ve B_\beta(\ve^{-\gamma}t)=
$$
$$
x\ve^\delta + \ve^\delta \int_0^{t} (c_+\1_{ X(z\ve^{-\gamma})\geq 0} -c_-\1_{X(z\ve^{-\gamma})< 0})
 |X(z\ve^{-\gamma})|^\alpha dz\ve^{-\gamma} +
$$
$$
+ \ve^\delta \ve\;  \ve^{-\gamma \beta} \ve^{\gamma \beta}B_\beta(\ve^{-\gamma}t)=
$$
$$
x\ve^\delta + \ve^\delta \ve^{-\gamma} \ve^{-\delta\alpha}\int_0^{t} (c_+\1_{\wt X(z )\geq 0} -c_-\1_{\wt X(z )< 0})
 |\wt X(z)|^\alpha  dz  +
$$
$$
+ \ve^{1+\delta  -\gamma \beta}  \wt B_\beta(t),
$$
where $\wt B_\beta(t)= \ve^{\gamma \beta} B_\beta(\ve^{-\gamma}t).$
\end{proof}
\begin{corl}\label{corl1}
Let $  X$ be a solution of \eqref{eq_main_00}, $X_\ve$ be a solution of \eqref{eq_main_0}.   Then
$$
\{X_\ve(t), t\geq 0\}\overset{d}{=}  \{\ve^{\frac{-1}{ (1-\alpha)\beta-1}}   X (\ve^{\frac{1-\alpha}{ (1-\alpha)\beta-1}}t), t\geq 0\}
$$
 if at least one of equations \eqref{eq_main_00} or  \eqref{eq_main_0} has a unique weak solution.
\end{corl}
\section{Asymptotic behavior of perturbed integral equation}\label{section3}

Let $c>0, x>0$. In this section we find
sufficient conditions   ensuring that a solution to the integral equation
\be\label{eq198}
y(t)=x+ c\int_0^t|y(s)|^\alpha ds +g(t)
 \ee
is equivalent as $t\to\infty$ to the solution of  
$$
z(t)= x+ c\int_0^t|z(s)|^\alpha ds 
 $$
if the function $g$ has comparatively small growth on infinity. 
I.e., we will find conditions on growth of $g$ that suffices
\be\label{eq_asympt_y_z}
y(t)\sim(c(1-\alpha)t)^{\frac{1}{1-\alpha}}, \ t\to\infty.
\ee

Let $X$ be a solution of \eqref{eq_main_00}. Under  natural integrability assumptions \cite{Khintchine,  Kono, Takashima89}, self-similar process $B_\beta$ satisfies the following growth condition
$$
\forall \delta>0 \ \ \lim_{t\to\infty}\frac{B_\beta(t)}{t^{\beta+\delta}}=0\ \ \mbox{a.s.}
$$

So, it is natural to expect that if $\omega\in\Omega$ is such that 
$X(t)=X(t,\omega)\to+\infty, t\to\infty$, then $X(t,\omega)\sim(c_+(1-\alpha)t)^{\frac{1}{1-\alpha}}, \ t\to\infty$ (or $X(t)\sim-(c_-(1-\alpha)t)^{\frac{1}{1-\alpha}}$ if $\lim_{t\to\infty}X(t)=-\infty$, respectively).
We  apply deterministic result on  equivalence \eqref{eq_asympt_y_z} to the 
stochastic equation \eqref{eq_main_0} in the next section. 

We consider the cases $\alpha\in(0,1)$ and $\alpha\in(-1,0]$ separately. In these cases the function $x^\alpha$ is increasing or decreasing, respectively. This has an effect on properties of the solution and the course of the proof.


{We will always assume that all functions below are measurable and locally bounded.}

\subsection{Case $\alpha\in(0,1)$}
In order to study the asymptotic behavior of $y(t)$ we need few simple  auxiliary lemmas.
\begin{lem}\label{lem2}
Let $a:\mbR\to\mbR$ be a non-decreasing, locally Lipschitz function of linear growth, $g_1, g_2:[0,T]\to \mbR$ be  bounded measurable functions,
$$
y_i(t) = x_i+ \int_0^t a(y_i(s))ds + g_i(t), t\in [0,T], \ i=1,2.
$$
Assume that $x_1\leq x_2,  g_1(t)\leq g_2(t), t\in [0,T].$ Then $y_1(t)\leq y_2(t), t\in [0,T].$
\end{lem}
Assumptions of the lemma yield that the solutions are unique and can be obtained by iterations. The  corresponding  inequality obviously is satisfied for iterations, so it is satisfied for their limits too.
\begin{remk}
The assumption on monotonicity of $a$ cannot be omitted.
\end{remk}
As a corollary of Lemma \ref{lem2} we get the following result.
\begin{lem}\label{lem3}
Let  $y_1(t)\geq \ve>0, t\in[0,T].$ Assume that $y_1(t), y_2(t), t\in[0,T]$ are such that
$$
y_1(t) \leq x_1+ c \int_0^t y^\alpha_1(s) ds + g_1(t), t\in [0,T], 
$$
and
$$
y_2(t) \geq x_2+ c \int_0^t y^\alpha_2(s)  ds + g_2(t), t\in [0,T], \ i=1,2,
$$
where $\alpha\in(0,1)$, $g_1, g_2:[0,T]\to \mbR$  are bounded measurable functions.

If $x_1\leq x_2,  g_1(t)\leq g_2(t), t\in [0,T]$, then $y_1(t)\leq y_2(t), t\in [0,T].$
\end{lem}
\begin{thm}\label{lem5}
Assume that $y(t)$ satisfies the equation
 $$
 y(t) = x + c \int_{0}^t |y(s)|^\alpha ds + g(t), t\geq 0,
 $$
  where $c>0, \alpha\in(0,1), g:[0,\infty)\to \mbR$ is a measurable locally bounded function.
Assume that $\lim_{t\to\infty} y(t)=+\infty$ and there exists a sequence $\{t_n\}$, $\lim_{n\to\infty} t_n=\infty$ and  a function $ h(t)=o(t^ {\frac{1}{1-\alpha}})$, $ t\to\infty,$
 such that
\be\label{eq_qqq}
|g(t)-g(t_n)|\leq h(t-t_n), \ t\geq t_n.
\ee
Then
$$
{y(t)}\sim {(c(1-\alpha)t)^{\frac{1}{1-\alpha}}}, t\to\infty.
$$


\end{thm}


 
\begin{remk} Assumption $\lim_{t\to\infty} y(t)=+\infty$ cannot be omitted. Indeed, $y(t)=0, t\geq 0, $ satisfies the equation with $x=0, g\equiv 0.$
The condition $
|g(t)-g(t_n)|\leq h(t-t_n), \ t\geq t_n
$ also cannot be replaced by $g(t)=o(t^ {\frac{1}{1-\alpha}}), t\to\infty$. Indeed, let $\gamma\in(0,\frac{1}{1-\alpha}), y(t)=t^\gamma,  {t_0=0, x=0}.$  Then
$$
y(t)=\int_0^t (s^\gamma)^\alpha ds -\int_0^t (s^\gamma)^\alpha ds+t^\gamma= 
$$
$$
=\int_{ {0}}^t |y(s)|^\alpha ds -\frac{t^{\alpha\gamma+1}}{\alpha\gamma+1}+t^\gamma=\int_{ {0}}^t |y(s)|^\alpha ds +g(t),
$$
where $g(t)=-\frac{t^{\alpha\gamma+1}}{\alpha\gamma+1}+t^\gamma$.

Since $\alpha\gamma+1<\frac{\alpha}{1-\alpha}+1=\frac{1}{1-\alpha}$, we have $g(t)=o(t^{\frac{1}{1-\alpha}})$ and 
$f(t)=o(t^{\frac{1}{1-\alpha}})$ as $t\to\infty.$
\end{remk}
\begin{proof}[Proof of Theorem \ref{lem5}] 
We have 
$$
y(t) = y(t_n) + c \int_{t_n}^t |y(s)|^\alpha ds + g(t)-g(t_n)\geq 
  y(t_n) + c \int_{t_n}^t |y(s)|^\alpha ds + h(t-t_n), t\geq t_n.
$$
Set $z_n(t)=y(t+t_n), x_n=y(t_n).$ Then
\be\label{eq257}
z_n(t)\geq x_n+ c\int_0^t |z_n(s)|^\alpha ds - h(t), \ t\geq 0.
\ee
For $a>0, t\geq 0$ denote $\wt z^{(a)}(t)=(1+a(1-\alpha)t)^{\frac{1}{1-\alpha}}.$
 The function  $\wt z^{(a)}$ satisfies the equation
$$
\wt z^{(a)}(t)=1+a\int_0^t  \wt z^{(a)}(s)^\alpha ds.
$$
Let $c_-<c.$
Then
$$
\wt z^{(c_-)}(t)=1+c_-\int_0^t \wt z^{(c_-)}(s)^\alpha ds=
$$
$$
1+c\int_0^t \wt z^{(c_-)}(s)^\alpha ds+(c_--c)\int_0^t  \wt z^{(c_-)}(s)^\alpha ds=
$$
$$
=1+c\int_0^t \wt z^{(c_-)}(s)^\alpha ds+
(c_--c)\left((1+c_-(1-\alpha)t)^{\frac{1}{1-\alpha}}-1\right)\leq 
$$
\be\label{eq276}
\leq K(c_-)+c\int_0^t \wt z^{(c_-)}(s)^\alpha ds-h(t), t\geq 0,
\ee
where $K(c_-)$ is a constant.

It follows from \eqref{eq257}, \eqref{eq276}, and Lemma \ref{lem3} that if $n$ is sufficiently large,
 then $y(t_n+t)=z_n(t)\geq {\wt z^{(c_-)}(t)}, t\geq 0.$ So
$$
 \overline{\lim_{t\to\infty}}\frac{y_n(t)}{(1+c_-(1-\alpha)t)^{\frac{1}{1-\alpha}}}\geq 1.
$$
Therefore
$$
\forall c_-<c\ \ \ \ \  \overline{ \lim_{t\to\infty}}\frac{y(t)}{(c(1-\alpha)t)^{\frac{1}{1-\alpha}}}\geq \frac{c_-^{\frac{1}{1-\alpha}}}{c^{\frac{1}{1-\alpha}}}.
$$
Similarly we get the inequality
$$
\forall c_+>c\ \ \ \ \  \underline{ \lim}_{t\to\infty}\frac{y(t)}{(c(1-\alpha)t)^{\frac{1}{1-\alpha}}}\leq \frac{c_+^{\frac{1}{1-\alpha}}}{c^{\frac{1}{1-\alpha}}}.
$$
Since $c_-<c$ and $c_+>c$ were arbitrary, this proves the theorem.

\end{proof}



\subsection{Case $\alpha\in(-1,0]$}
In this subsection we  assume that solution of \eqref{eq198} is positive  for all $t\geq0$.
\begin{lem}\label{lemma2.1}
Let $\ve\in(0,1),  a>1, \Delta\in (0,\frac{1+\alpha}{-\alpha})\cap (0,  \frac1\ve-1)  $ be arbitrary. 
There exists $\delta=\delta(\ve, a, \Delta)>0$ depending only on $\ve, a, \Delta$ such that if
$$
  |g(t)|\leq \delta t^{\frac{1}{1-\alpha}}, \ t\in[a^n,a^{n+1}],
$$
for some $n\geq 1$ and $y$ is a solution of \eqref{eq198} such that
\be\label{eq1.2}
(1-\ve)(c(1-\alpha)a^n)^{\frac1{1-\alpha}}<y(a^n)<(1+\ve)(c(1-\alpha)a^n)^{\frac1{1-\alpha}},
\ee
then for all $t\in [a^n,a^{n+1}]$ we have
\be\label{eq1.3}
(1-\ve(1+\Delta))(c(1-\alpha)a^n)^{\frac1{1-\alpha}}<y(t)<(1+\ve)(c(1-\alpha)a^{n+1})^{\frac1{1-\alpha}},
\ee
and
\be\label{eq1.4}
(1-\ve)(c(1-\alpha)a^{n+1})^{\frac1{1-\alpha}}<y(a^{n+1}).
\ee
\end{lem}
\begin{proof}
The function $x^\alpha, x>0$ is decreasing. So, if
$$
0<y_-(t)<y(t)<y_+(t), \ t\in [a^n, a^{n+1}],
$$
then  {for all $t\in [a^n, a^{n+1}]$}
$$
y(t)=y(a^n)+\int_{a^n}^tcy^\alpha(s)ds + g(t)-g(a^n)<
$$
$$
 y_+(a^n)+\int_{a^n}^tcy_-^\alpha(s)ds + |g(t)|+|g(a^n)|\leq
$$
$$
y_+(a^n)+c \max_{s\in[a^n,t]}y_-^\alpha(s)(t-a^n)+2\delta t^\frac1{1-\alpha} \leq
$$
\be\label{eq1.5}
y_+(a^n)+c a^n(a-1) \max_{s\in[a^n,a^{n+1}]}y_-^\alpha(s)+2\delta a^\frac{n+1}{1-\alpha}.
\ee
Similarly
\be\label{eq1.6}
y(t)>  y_-(a^n)+c a^n(a-1) \min_{s\in[a^n,a^{n+1}]}y_+^\alpha(s)-2\delta a^\frac{n+1}{1-\alpha}.
\ee
By $y_-$ and $y_+$ denote the left hand side and the right hand side of \eqref{eq1.3}, correspondingly. To prove the lemma it suffices to
show that there exists $\delta>0$ such that  the upper (and the lower) bound of $y$ from inequalities \eqref{eq1.5} (respectively \eqref{eq1.6})
 is less than the right hand side of \eqref{eq1.3}  (is greater than the left hand side of \eqref{eq1.3} and \eqref{eq1.4}).

Let us check the upper bound only. The lower bound can be proved similarly.

Let $t\in[a^{n},a^{n+1}].$ We have to verify  
$$
(1+\ve)\left(c(1-\alpha)a^n\right)^{\frac{1}{1-\alpha}}+c a^n (a-1)
\left(  (1-\ve(1+\Delta))(c(1-\alpha)a^n )^{\frac{1}{1-\alpha}}\right)^\alpha
+2\delta a^{\frac{n+1}{1-\alpha}}<
$$
$$
(1+\ve)\left(c(1-\alpha)a^{n+1}\right)^{\frac{1}{1-\alpha}}
$$
or the following equivalent inequality
\be\label{eq1.8}
 \frac{(1-\ve(1+\Delta))^\alpha(a-1)}{1-\alpha}+\frac{2\delta a^{\frac{1}{1-\alpha}}}{(c(1-\alpha) )^{\frac{1}{1-\alpha}}}<(1+\ve)(a^{\frac1{1-\alpha}}-1)
\ee
for some fixed $\delta =\delta(\ve, a,\Delta)>0$.

Since $\alpha\in(-1,0]$, the mean value theorem yields inequalities
\be\label{eq3.1}
\forall x\in(0,1)\ \ (1-x)^\alpha\leq1-\alpha x;\ \ \ \ \forall x>0\ \ x^{\frac1{1-\alpha}} -1 \geq \frac {x-1}{1-\alpha}.
\ee
It follows from \eqref{eq3.1} that to prove \eqref{eq1.8} it is sufficient to check inequality 
\be\label{eq3.2}
 \frac{(1-\ve(1+\Delta)\alpha)(a-1)}{1-\alpha}+\frac{2\delta a^{\frac{1}{1-\alpha}}}{(c(1-\alpha) )^{\frac{1}{1-\alpha}}}<\frac{(1+\ve)(a-1)}{1-\alpha}
\ee
for some $\delta>0.$

Since $|(1+\Delta)\alpha|<1$ by the assumptions of the lemma, inequality \eqref{eq3.2} is true for sufficiently small $\delta.$ Lemma \ref{lemma2.1} is proved.
\end{proof}
The following statement is a simple corollary of  Lemma \ref{lemma2.1}.
\begin{lem}\label{lem2.2}
Let $\ve\in(0,1),  a>1, \Delta\in (0,\frac{1+\alpha}{-\alpha})\cap (0,  \frac1\ve-1)  $ be arbitrary. Assume that
 
(i) $\delta=\delta(\ve, a, \Delta)>0$ is  selected 
from Lemma \ref{lemma2.1}, 

(ii)  inequality \eqref{eq1.2} is true for 
some $n_0$,

(iii)
    $\ \ \ 
|g(t)|\leq \delta t^{\frac1{1-\alpha}}, \ t\geq a^{n_0}.
$

Then
\be\label{eq3.3}
(1-\ve(1+\Delta))\leq\liminf_{t\to\infty}\frac{y(t)}{\left(c(1-\alpha)t\right)^{\frac{1}{1-\alpha}}}\leq
\ee
$$
\limsup_{t\to\infty}\frac{y(t)}{\left(c(1-\alpha)t\right)^{\frac{1}{1-\alpha}}}\leq (1+\ve)a^{\frac1{1-\alpha}}.
$$
\end{lem}
Indeed, it follows from Lemma \ref{lemma2.1} that for all $n\geq n_0$ we have \eqref{eq1.3}.
So,
$$
\limsup_{t\to\infty}\frac{y(t)}{\left(c(1-\alpha)t\right)^{\frac{1}{1-\alpha}}}\leq \limsup_{n\to\infty}\frac{\sup_{t\in[a^n, a^{n+1}]}y(t)}{\left(c(1-\alpha)a^n\right)^{\frac{1}{1-\alpha}}}
\leq 
$$
$$
\limsup_{n\to\infty}\frac{(1+\ve)(c(1-\alpha)a^{n+1})^{\frac1{1-\alpha}}}{\left(c(1-\alpha)a^n\right)^{\frac{1}{1-\alpha}}} = (1+\ve)a^{\frac1{1-\alpha}}.
$$
The inequality for $\lim\inf$ is proved similarly.
\begin{thm}\label{theorem3}
Assume that $\lim_{t\to\infty}\frac{g(t)}{t^{\frac{1}{1-\alpha}}}=0$ and  the solution of \eqref{eq198} is such that
$$
\liminf_{t\to\infty}y(t)>0.
$$
Then
 $$
\lim_{t\to\infty}\frac{y(t)}{ (c(1-\alpha) t)^{\frac1{1-\alpha}}}=1.
$$
\end{thm}

\begin{proof}
Let $\ve>0, \Delta\in (0,\frac{1+\alpha}{-\alpha})\cap (0,  \frac1\ve-1), a>1$ be arbitrary, and $\delta =\delta(\ve, a,\Delta)$ be from    Lemma \ref{lemma2.1}. Select  $n$ such that
$$
\inf_{t\geq a^n}y(t)>0, \ \ \sup_{t\geq a^n}\frac{|g(t)|}{t^{\frac1{1-\alpha}}}\leq \frac{\delta}2.
$$
Let
$$
x(t)=(c(1-\alpha)a^n)^{\frac1{1-\alpha}}+c\int_{a^n}^t x^{\alpha}(s)ds+g(t)-g(a^n).
$$
Note that $x(t)>0, t\geq a^n$ by Lemma \ref{lemma2.1}.

It 
follows from Lemma \ref{lem2.2}  that
$$
(1-\ve(1+\Delta))\leq\liminf_{t\to\infty}\frac{x(t)}{\left(c(1-\alpha)t\right)^{\frac{1}{1-\alpha}}}\leq
\limsup_{t\to\infty}\frac{x(t)}{\left(c(1-\alpha)t\right)^{\frac{1}{1-\alpha}}}\leq (1+\ve)a^{\frac1{1-\alpha}}.
$$ 
Assume that $y(a^n)\geq x(a^n).$ Then by the comparison theorem $ y(t)\geq x(t)>0, t\geq a^n.$
Since   $ \alpha<0,$  we have 
$$
x(t)-y(t)\leq x(a^n)-y(a^n)+\int_{a^n}^t(x^{\alpha}(s)-y^{\alpha}(s))ds\leq x(a^n)-y(a^n), t\geq a^n.
$$
Similarly, if $y(a^n)\leq x(a^n),$ then $
x(t)-y(t)\geq  x(a^n)-y(a^n), t\geq a^n.
$
Hence, in any case
$$
|x(t)-y(t)|\leq  |x(a^n)-y(a^n)|, t\geq a^n.
$$
Therefore
$$
(1-\ve(1+\Delta))\leq\liminf_{t\to\infty}\frac{y(t)}{\left(c(1-\alpha)t\right)^{\frac{1}{1-\alpha}}}\leq
\limsup_{t\to\infty}\frac{y(t)}{\left(c(1-\alpha)t\right)^{\frac{1}{1-\alpha}}}\leq (1+\ve)a^{\frac1{1-\alpha}}.
$$ 
Since $\ve>0, \Delta\in (0,\frac{1+\alpha}{-\alpha})\cap (0,  \frac1\ve-1),$ and $a>1$ were arbitrary, the theorem is proved.
\end{proof}
The following result follows from the course of Theorems' \ref{lem5} and \ref{theorem3} proof.
\begin{lem}\label{remk33}
{ Let  $h(t)=o(t^{\frac1{1-\alpha}}), t\to\infty$ be a non-negative function, and $y$ be a solution of \eqref{eq198}, where $\alpha\in(-1,1)$.
There is $R=R(\alpha, g, h)>0$ such that if 
$y(t_0)\geq R$ and $|g(t-t_0)|\leq h(t), t\geq 0, $ then
$\lim_{t\to\infty}{y(t)}=+\infty.$
}
\end{lem}

\section{Proof of the main results}\label{sec4}

\begin{thm}\label{thm2}
Let $\wt X(t), t\geq 0,$ be a  solution to   SDE 
$$
\wt X(t)=\int_0^t (c_+\1_{\wt X(s)\geq 0} -c_-\1_{\wt X(s)< 0})
 |\wt X(s)|^\alpha ds +  B(t), t\geq 0,
$$
where  $c_\pm>0,$ $B$ is a c\'{a}dl\'{a}g stochastic  process.

Suppose that  

1)  $\alpha\in(0,1)$ and there exists    a function $h(t)=o(t^ {\frac{1}{1-\alpha}}), t\to\infty,$  and a (random) sequence $\{t_n\}$, $\lim_{n\to\infty} t_n=\infty,$
such that
$$
|B(t)-B(t_n)|\leq h(t-t_n), \ t\geq t_n \ \mbox{a.s.}
$$

or

2) $\alpha\in(-1,0]$ and $B(t)=o(t^{\frac1{1-\alpha}}), t\to\infty$ a.s.

Then

$$
\wt X(t)\sim (c_+(1-\alpha)t)^{\frac{1}{1-\alpha}} \ \ \mbox{for a.a.}\  \omega\in\{\lim_{t\to\infty}\wt X(t)=+\infty\},
 $$
$$
\wt X(t)\sim -(c_-(1-\alpha)t)^{\frac{1}{1-\alpha}} \ \ \mbox{for a.a.}\  \omega\in\{\lim_{t\to\infty}\wt X(t)=-\infty\}.
$$
 \end{thm}
The result follows from Theorems \ref{lem5} and \ref{theorem3}  
and their natural modifications to the case   $c<0$.
Note that $\wt X$ has c\'{a}dl\'{a}g trajectories, so sets $ \{\lim_{t\to\infty}\wt X(t)=+\infty\}$ and $\{\lim_{t\to\infty}\wt X(t)=-\infty\}$ are measurable. 
To solve the selection problem for a process $X_\ve$ defined in \eqref{eq_main_0} we need the following result.
\begin{lem}\label{lem55}
Assume that a locally bounded  function $f:[0,\infty)\to\mbR$ is such that $f(t)\sim Kt^A, t\to\infty,$ where $A>0.$

Set $g_n(t)=n^{-1}f(n^{1/A}t).$ Then for any $T>0$ 
we have
$$
\lim_{n\to\infty} \sup_{s\in[0,T]}|g_n(s)-Ks^A|=0.
$$
\end{lem}

\begin{proof}
Let $\ve>0, \delta>0$ be arbitrary. 
Select $R>0$ such that $|f(s)|\leq R(1+s^A),\ s\geq 0.$
Then
$$
\sup_{s\in[0,\delta]}|g_n(s)|= \sup_{s\in[0,\delta]}|\frac{ f(n^{1/A}s)}{n}|\leq  \sup_{s\in[0,\delta]}
\frac{R(1+(n^{1/A}s)^A)}n =
\frac{R}{n}+R\delta^A.
$$
Select $t_0$ such that
$$
|\frac{f(t)}{Kt^A}-1|<\ve,\ t\geq t_0.
$$
Then for sufficiently large $n:$

$$
\sup_{s\in[\delta,T]}|g_n(s)-Ks^A|=\sup_{s\in[\delta,T]}\left({Ks^A} \left|\frac{g_n(s)}{Ks^A}-1\right|\right)\leq 
KT^A \sup_{s\in[\delta,T]}|\frac{g_n(s)}{Ks^A}-1|=
$$
$$
KT^A \sup_{s\in [\delta,T]}|\frac{ f(n^{1/A}s)}{ K(n^{1/A}s)^A}-1|\leq KT^A \sup_{n^{1/A}s\geq n^{1/A}\delta }|\frac{ f(n^{1/A}s)}{ K(n^{1/A}s)^A}-1| 
=
$$
$$
KT^A \sup_{z\geq n^{1/A}\delta }|\frac{ f(z)}{ K z^A}-1| <KT^{1/A}\ve.
$$

Thus
$$
\limsup_{n\to\infty} \sup_{s\in[0,T]}|g_n(s)-Ks^A|\leq R\delta^A + KT^{1/A}\ve.
$$

Since $\ve>0$  and $\delta>0$ are arbitrary, the lemma is proved.
\end{proof}
\begin{proof}[Proof of Theorem \ref{thm_main}]
Let $X$ be a solution to \eqref{eq_main_00}. 
Theorem  \ref{thm2} yields that for a.a.  $\omega\in\{ \lim_{t\to\infty}X(t)=+\infty\}$  we have 
$X(t)\sim {(c_+(1-\alpha)t)^{\frac{1}{1-\alpha}}}, t\to\infty$.

Set 
$f=X, n=\ve^{\frac{1}{ (1-\alpha)\beta-1}},  A=\frac{1}{1-\alpha}, K=(c_+(1-\alpha))^{\frac{1}{1-\alpha}}$ in Lemma \ref{lem55}.
Assumption $\alpha+\beta^{-1}>1$ yields $(1-\alpha)\beta-1 <0.$

 Denote $\wt X_\ve(t):=\ve^{\frac{-1}{ (1-\alpha)\beta-1} }  X (\ve^{\frac{1-\alpha}{(1-\alpha)\beta-1}}t), t\geq 0.$
So, it follows from Lemma \ref{lem55} that  for any $T>0$ and a.a. $\omega\in \{ \lim_{t\to\infty}  X(t) =+\infty\}$ we have the uniform convergence
$$
\lim_{\ve\to0+} \sup_{s\in[0,T]}|\wt X_\ve(t)-(c_+(1-\alpha)t)^{\frac{1}{1-\alpha}}|=0.
$$



Similarly, we have
$$
\lim_{\ve\to0+} \sup_{s\in[0,T]}|\wt X_\ve(t)+(c_-(1-\alpha)t)^{\frac{1}{1-\alpha}}|=0
$$
 for a.a.  $\omega$
 such that $\lim_{t\to\infty}X(t)=-\infty.$

Therefore, for a.a. $\omega\in \left\{\lim_{t\to\infty}  X(t) =+\infty \ \mbox{ or }\ \lim_{t\to\infty}  X(t) =-\infty\right\}$
we have the uniform convergence
$$
\lim_{\ve\to0+} \sup_{s\in[0,T]}\left|\wt X_\ve(t)-((1-\alpha)t)^{\frac{1}{1-\alpha}}\left(c_+^{\frac{1}{1-\alpha}}\1_{\lim_{t\to\infty}  X(t)
 =+\infty }- \right.\right.
$$
 \be\label{eq846}
\ \ \ \ \ \ \ \ \ \ \ \ \ \ \ \ \ \ \ \ \ \left.\left. c_-^{\frac{1}{1-\alpha}}\1_{\lim_{t\to\infty}  X(t) =-\infty }\right) \right|=0.
\ee
It follows from Lemma \ref{remk33}  that under assumptions   of the Theorem 
$$
\Pb\left(
\{\lim_{t\to\infty}|X(t)|=\infty\}\bigtriangleup 
\left(\{\lim_{t\to\infty}X(t)=+\infty\}\cup \{\lim_{t\to\infty}X(t)=-\infty\}\right)
\right)=0.
$$

It follows from uniqueness of the solution and Corollary \ref{corl1} that 
the distribution of $X_\ve(t), t\geq 0,$ equals the distribution of
 $\wt X_\ve(t), t\geq 0.$

This and the convergence \eqref{eq846} yields the proof.
 
\end{proof}

\section{Examples}\label{Examples}
L\'{e}vy stable processes and a fractional Brownian motion satisfy assumptions 1) and 3) of Theorem \ref{thm_main}, see Remarks \ref{remUniq} and \ref{remkCond3}. In this section we show that these processes also satisfy assumption 2) of the Theorem.

\begin{example}[Symmetric L\'{e}vy stable noise]\label{ex1}
Let $B_\beta(t), t\geq 0$ be a symmetric $1/\beta$-stable process. Then  $B_\beta$ is self-similar process with parameter $\beta.$

It was mentioned in Remark \ref{remUniq}  that there exists a unique  solution to \eqref{eq_dod}
if  $\alpha+1/\beta>1.$  
This solution is a strong Markov process \cite{TanakaTsuchiya74, Portenko1994}.
Condition 3) of Theorem \ref{thm_main}
is also satisfied, see Remark \ref{remkCond3}.

 Let us verify that condition 2) of Theorem \ref{thm_main}
is satisfied.
\begin{remk}
We give proofs only for $\alpha\in(0,1).$ The case $\alpha\in(-1,0]$ is  simpler.
\end{remk}
Let $\ve>0$ be such that $\alpha+(\beta+\ve)^{-1}>1.$ Set $h(t)=1+t^{\beta+\ve}.$ 

It follows from \cite{Khintchine} that $\lim_{t\to\infty}B_{\beta}(t)/h(t)=0$ a.s. 

It is easy to see that 
$$
\gamma:=\Pb(|B_\beta(t)|>h(t), \ t\geq 0)>0.
$$
 
To prove
 \be\label{eq_X_infty1}
\lim_{t\to\infty}|\wt X(t)|=\infty\ \ \mbox{ a.s.},
 \ee
 it suffices to show that
\be\label{eq_X_infty}
\forall K>0\ \Pb\left(\exists t_0\ \forall t\geq t_0\ \ \ |\wt X(t)|\geq K\right)=1.
\ee
Let $K$ be fixed. It follows from  Lemma \ref{remk33}  that 
there exists $M=M(K)$ such that

{\it  if $\omega$ and $t_0$ are such that  $|\wt X(t_0)|\geq M$ and $|B_\beta(s+t_0)-B_\beta(t_0)|<h(s), s\geq 0,$ 
then $|\wt X(t)|\geq K, t\geq t_0.$}
\newline
Set
\newline
 $\tau_1:=\inf\{t\geq 0 : |\wt X(t)|\leq K\},$ $\sigma_n:=\inf\{t\in \mbN\cap[\tau_n,\infty) : |\wt X(t)|\geq M\},$ $ \tau_{n+1}:=\inf\{t\geq \sigma_n : |\wt X(t)|\leq K\}, n\geq 1.$ 

Then  {by the strong Markov property,}
$$
X_n:=\begin{cases}
         \wt X(\tau_n), \ \ \tau_n<\infty,\\
         \infty, \ \ \tau_n=\infty.
\end{cases}
$$
 is a homogeneous Markov chain. 
 
 It is obvious that $\Pb(\sigma_n<\infty | \tau_n<\infty)=1.$ Hence
 $$
 \sup_{|x|\leq K}\Pb_x(X_{ 1}=\infty | X_0=x)\geq \Pb(|B_\beta(s+\sigma_1)-B_\beta(\sigma_1)|>h(s), s\geq 0)=\gamma>0.
 $$
 Therefore
$
\Pb(X_n<\infty)\leq (1-\gamma)^n.
$

So 
$$
\Pb(\exists n_0 \ : \ X_{n_0}=\infty)=1.
$$
This yields \eqref{eq_X_infty} and hence \eqref{eq_X_infty1}. 

Thus, we have proved that all  conditions of Theorems \ref{thm2} and \ref{thm_main} are satisfied for $1/\beta$-stable  L\'{e}vy processes if $\alpha+1/\beta>1.$ So
for a.a. $\omega$ we have either
$\wt X(t)\sim(c_+(1-\alpha)t)^{\frac{1}{1-\alpha}}$ or
 $\wt X(t)\sim -(c_-(1-\alpha)t)^{\frac{1}{1-\alpha}}$ as $t\to\infty.$ Moreover
  $X_\ve$ converges in distribution as $\ve\to0$ to 
 $$
((1-\alpha)t)^{\frac{1}{1-\alpha}}\left({c_+^{\frac{1}{1-\alpha}}}\1_{\{\lim_{t\to\infty}\wt X(t)=+\infty\}}- {c_-^{\frac{1}{1-\alpha}}}\1_{\{\lim_{t\to\infty}\wt X(t)=-\infty\}}\right).  
$$
\end{example}

\begin{example}[Fractional Brownian motion]\label{ex2}
 Let $\{B_{H}(t)\}$ be a fractional Brownian motion with Hurst parameter $H\in (0,2)$. Then $B_H$
 is the self-similar process with index $H.$ 
 Assume that $\alpha>0$ and $\alpha+H^{-1}>1.$ Then, see Remarks \ref{remUniq} and \ref{remkCond3} ,
 assumptions 1) and 3) of Theorem \ref{thm_main} are satisfied.
 
  Let us verify that condition 2) of Theorem \ref{thm_main}.
  
  The solution of the equation is not a Markov process, so we cannot use method of the previous example.
  
  Denote by $X_x(t), t\geq 0,$ the solution of the SDE
  $$
  X_x(t)=x+\int_0^t (c_+\1_{X_x(s)\geq 0} -c_-\1_{X_x(s)< 0})
 |X_x(s)|^\alpha ds +  B_H(t),\ t\geq 0.
  $$
  Let $x<y$. Since the function $x\to(c_+\1_{x\geq 0} -c_-\1_{x(s)< 0}) |x|^\alpha$ is increasing,
  the derivative $\frac{d(X_y(t)-X_x(t))}{dt}$ is positive and the function $t\to X_y(t)-X_x(t)$  is increasing too. 
  
  The process $X_x(t), t\geq 0,$ is defined up to a set of null measure. Set
  $\tilde X_x(t):=\limsup{y\downarrow x} X_y(t)$. It is easy to see that 
  $\tilde X_x(t)$ is a modification of $X_x(t)$ which is measurable in $(t,x,\omega)$.
  Further we consider only this modification and omit tilde in the notations.
\begin{lem}\label{lem_x^*}
For a.e. $\omega$ there exists a unique $x^*=x^*(\omega)$ such that
$$
\forall x>x^*\ \ \ \lim_{t\to\infty}X_x(t,\omega)=+\infty,
$$
$$
\forall x<x^*\ \ \ \lim_{t\to\infty}X_x(t,\omega)=-\infty.
$$
\end{lem}
\begin{proof}
Let $h$ and $\{t_n\}$ satisfy the assumptions of the theorem, the existence see in Remark \ref{remkCond3}.
Select $R>0$ from Lemma \ref{remk33} such that if  
$$
y_x(t)=x+ \int_0^t (c_+\1_{y_x(s)\geq 0} -c_-\1_{y_x(s)< 0}) 
 |y_x(s)|^\alpha ds
+g(t),\ \  t\geq 0,
$$
with $|x|>R$ and $|g(t)|\leq h(t), t\geq 0, $ then $y_x(t), t\geq 0,$ never changes 
the sign and $\lim_{t\to\infty} |y_x(t)|=\infty.$

It follows from the monotonicity of $X_x(t)$ in $x$ that to prove the Lemma it is sufficient
to show that for every pair $x_1<x_2$ 
$$
\Pb(\lim_{t\to\infty} X_{x_1}(t)=-\infty \ \ \mbox{ or }\ \ \lim_{t\to\infty} X_{x_2}(t)=+\infty)=1.
$$
Assume the converse. Then 
there exists
 $\omega$  such that  
 $$
 \forall n\geq 1\ \ |B_\beta(t)- B_\beta(t_n)|\leq h(t-t_n), t\geq t_n,
 $$
 and
 $$
 \forall n\geq 1\ \ |X_{x_1}(t_n)| \leq R,\ |X_{x_2}(t_n)| \leq R,
$$
where $R$ is as before. 

Hence for this $\omega$
$$
\exists \ve_0>0\ \forall n\geq 1\ \ \forall t\in[0,1]\ \ \frac{d(X_{x_2}(t_n+t)-X_{x_1}(t_n+t))}{dt}\geq \ve_0.
$$

So
$$
  \forall n\geq 1\    X_{x_2}(t_n+1)-X_{x_1}(t_n+1) \geq X_{x_2}(t_n )-X_{x_1}(t_n )
+ \ve_0.
$$
Without loss of generality we may assume that $t_{n+1}-t_{n}\geq 1, n\geq 1.$

Therefore
$$
    X_{x_2}(t_n )-X_{x_1}(t_n ) \geq x_2-x_1 +n  \ve_0\to\infty,\ \ n\to\infty.
$$
This contradiction proves the Lemma.
\end{proof}
It follows from  Lemma \ref{lem_x^*} that there is a at most countable set $A\subset\mbR$
such that
$$
\forall x\in A\ \ \ \Pb(\lim_{t\to\infty} |X_x(t)|\neq \infty) >0.
$$
Fractional Brownian motion has stationary increments. So, if for some $x_0\in\mbR,\ t>0$
$$
\Pb(X_{x_0}(t)\notin A ) =1,
$$
then
\be\label{eq1981}
\Pb(\lim_{t\to\infty}|X_{x_0}(t)|=\infty) =1.
\ee
By Girsanov's theorem \cite{NualartOuknine}, the distribution of $X_{x_0}(t)$
is absolute continuous for any $x_0\in\mbR,\ t>0$. Hence
\eqref{eq1981} is satisfied for any $x_0\in\mbR$. So,  Theorem \ref{thm_main} is applicable in the case when $\alpha>0$ and  a self-similar process $B_\beta$ is an fBm.
\end{example}

\end{document}